\documentclass[12pt]{amsart}
\usepackage{amscd}
\usepackage{mathrsfs}
\usepackage{a4wide}
\usepackage{color}
\usepackage[utf8x]{inputenc}

\usepackage{amsmath,amsthm}
\usepackage{amssymb}
\usepackage{enumerate}
\usepackage{todonotes}
\usepackage{graphicx}
\usepackage{cancel}
\usepackage{xcolor}
\usepackage[arrow, matrix, curve, cmtip]{xy}
\usepackage{tikz}
\usepackage{geometry}
\usepackage[shortlabels]{enumitem}

\newtheorem{theorem}{Theorem}[section]

\newtheorem{corollary}[theorem]{Corollary}

\newtheorem{lemma}[theorem]{Lemma}

\newtheorem{proposition}[theorem]{Proposition}

\newtheorem{question}[theorem]{Question}

\theoremstyle{definition}
\newtheorem{definition}[theorem]{Definition}
\newtheorem{remark}[theorem]{Remark}

\date{}
\begin{document}

\title[Ramsey theory for monochromatically well-connected subsets]{Ramsey theory for monochromatically well-connected subsets}
\author{Jeffrey Bergfalk}

\address{Centro de Ciencas Matem\'{a}ticas\\
UNAM\\
A.P. 61-3, Xangari, Morelia, Michoac\'{a}n\\
58089, M\'{e}xico}

\email{jeffrey@matmor.unam.mx}
\thanks{{\it Date}: \today.\newline
{\it 2010 MSC}: 03E02, 03E55\newline
{\it Key words and phrases.} well-connected, highly connected, Ramsey theory, Mitchell forcing, weakly compact.}

\begin{abstract} We define \emph{well-connectedness}, an order-theoretic notion of largeness whose associated partition relations $\nu\to_{wc}(\mu)_\lambda^2$ formally weaken those of the classical Ramsey relations $\nu\to(\mu)_\lambda^2$. We show that it is consistent that the arrows $\to_{wc}$ and $\to$ are, in infinite contexts, essentially indistinguishable. We then show, in contrast, that in Mitchell's model of the tree property at $\omega_2$, the relation $\omega_2\to_{wc}(\omega_2)_\omega^2$ does hold, and that the consistency strength of this relation holding is precisely a weakly compact cardinal. These investigations may be viewed as augmenting those of \cite{BHS}, the central arrow of which, $\to_{hc}$, is of intermediate strength between $\to_{wc}$ and the Ramsey arrow $\to$.
\end{abstract}
\maketitle

The recent \emph{Ramsey theory for highly connected monochromatic subgraphs} \cite{BHS} introduces a graph-theoretic notion of largeness and studies the associated partition relations, denoted therein by the modified arrow notation $\nu\to_{hc}(\mu)_{\lambda}^2$. Graphs large in this sense are termed \emph{highly connected}, with complete graphs as the most obvious examples. Hence the arrow $\to_{hc}$ weakens the classical Ramsey arrow $\to$, though how much it does so tends to depend on assumptions supplementary to the $\mathsf{ZFC}$ axioms. In particular, it remains an open question at the time of this writing whether it is consistent, modulo large cardinal assumptions, that $\omega_2\to_{hc}(\omega_2)_\omega^2$.

In the following, we describe an order-theoretic notion of largeness termed \emph{well-connectedness}, whose associated partition relations $\to_{wc}$ weaken those of $\to_{hc}$, and hence those of $\to$, yet further. We show that this weakening is mild in the sense that these three sorts of partition relations are consistently identical in infinite settings. We show on the other hand that in Mitchell's model of the tree property at $\omega_2$, the relation $\omega_2\to_{wc}(\omega_2)_\omega^2$ holds while $\omega_2\to_{hc}(\omega_2)_\omega^2$ fails. We conclude by deducing that the consistency strength of the former relation is exactly a weakly compact cardinal.

Aside from the aforementioned arrows, our notations and conventions are standard. By the \emph{size} of a graph we mean the cardinality of its vertex-set, and though for readability we write $\omega$ and $\omega_2$, for example, our interest throughout is in the partition relations of cardinals. We mainly follow \cite{cummings_handbook} in our account of Mitchell's forcing;\footnote{It is perhaps worth noting that Mitchell in \cite[p. 41]{mitchell1} credits the variant of this forcing targeting the tree property to Silver. This is the variant we employ.} readers are directed to Section 5 therein for Easton's Lemma, or for the definition and basic properties of the projection of a forcing poset. Add$(\kappa,\lambda)$ denotes the usual forcing poset for adding $\lambda$ many subsets of $\kappa$, in which conditions all have size less than $\kappa$.

The core notion in \cite{BHS} is the following: 
\begin{definition}\label{highlyconnected}
 A graph $G=(V,E)$ is \emph{highly connected} if it remains connected after the deletion of any fewer than $|V|$ vertices. Write $\nu\to_{hc}(\mu)^2_\lambda$ if and only if for every coloring of the edges of the complete graph on $\nu$ by $\lambda$ many colors there exists some size-$\mu$ monochromatic subgraph which is highly connected.
\end{definition}
The only highly connected graph on a finite set $m$ of vertices is the complete one; hence for any finite $\ell$ and $m$ the relation $\nu\to_{hc}(m)^2_\ell$ corresponds precisely to the classical Ramsey relation $\nu\to(m)^2_\ell$. However, unlike the classical relation $\omega\to(\omega)_\ell^2$, the relation $\omega\to_{hc}(\omega)_\ell^2$ generalizes to the uncountably infinite:
\begin{proposition}[\cite{BHS}] \label{finite} Let $\ell$ be a positive integer. Then the relation $\mu\to_{hc}(\mu)_\ell^2$ holds for any infinite cardinal $\mu$.
\end{proposition}
The situation for infinitely many colors is more complicated. A variant of the Sierpi\'{n}ski coloring of \cite{sierpinski}, for example, witnesses the following:
\begin{proposition}[\cite{BHS}] \label{sierpinski} Let $\lambda$ and $\mu\geq 2^\lambda$ be infinite cardinals. Then $\mu\not\to_{hc}(\mu)_\lambda^2$.
\end{proposition}
In particular, $\omega_1\not\to_{hc}(\omega_1)_{\omega}^2$, and $\omega_2$ is the least cardinal $\mu$ for which the relation $\mu\to_{hc}(\mu)_{\omega}^2$ may possibly hold. From the assumptions of the continuum hypothesis and the existence of a weakly compact cardinal, respectively, models are constructed in \cite{BHS} of the relations $\omega_2\to_{hc}(\omega_1)_{\omega}^2$ and $2^{\omega_1}\to_{hc}(2^{\omega_1})_{\omega}^2$. However, it remains an open question at the time of this writing whether the relation $\omega_2\to_{hc}(\omega_2)_{\omega}^2$ is consistent (modulo large cardinal assumptions) with the $\mathsf{ZFC}$ axioms.

We turn now to the following closely related property and relation.
\begin{definition}\label{wellconnected} Given a coloring $c:[\nu]^2\rightarrow\lambda$, say $X\subseteq\nu$ is \emph{well-connected in the color $i$} if for every $\alpha<\beta$ in $X$ there exists a finite path from $\alpha$ to $\beta$, \begin{itemize}
\item all the edges of which are colored $i$, and
\item all the vertices of which are greater than or equal to $\alpha$.
\end{itemize}
Note that we do not require those vertices to lie in $X$. Write $$\nu\rightarrow_{wc}(\mu)_\lambda^2$$ if and only if for every $c:[\nu]^2\rightarrow\lambda$ there exists an $X\subseteq\nu$ of cardinality $\mu$ which is well-connected in some color $i\in\lambda$.
\end{definition}
Observe that the key object $X$ in the above definition does not itself involve a choice of edges. In other words, while Definition \ref{highlyconnected} describes a potential property of graphs (i.e., of subsets of $[\nu]^2$), Definition \ref{wellconnected} describes a potential property of sets $X\subseteq\nu$. Nevertheless, the partition relations associated to these two properties, as well as those of the classical Ramsey arrow, can be, in infinite contexts, essentially indistinguishable.
\begin{theorem}\label{inL}
Suppose that $\mathrm{V}=\mathrm{L}$ and that $\mu$ is a regular cardinal. \begin{itemize}
\item If $\nu$ is finite then $$\nu\to(\mu)^2_\lambda\text{ if and only if }\nu\to_{hc}(\mu)^2_\lambda$$
\item If $\nu$ is infinite and $\lambda$ is finite then $$\nu\to_{hc}(\mu)^2_\lambda\text{ if and only if }\nu\to_{wc}(\mu)^2_\lambda$$
\item If $\nu$ is infinite and $\lambda$ is infinite then $$\nu\to(\mu)^2_\lambda\text{ if and only if }\nu\to_{hc}(\mu)^2_\lambda\text{ if and only if }\nu\to_{wc}(\mu)^2_\lambda$$
\end{itemize}
\end{theorem}
The theorem bundles together implications each of which follows from a weaker hypothesis than $\mathrm{V}=\mathrm{L}$; several of these simply hold in $\mathsf{ZFC}$. We list these implications separately in four lemmas. The proof of Theorem \ref{inL} will then consist briefly of applying these lemmas to $\mathrm{L}$.
\begin{lemma}\label{positive} Any positive relation $\nu\to(\mu)^2_\lambda$ implies $\nu\to_{wc}(\mu)^2_\lambda$, which implies $\nu\to_{wc}(\mu)^2_\lambda$ in turn.
\end{lemma}
\begin{proof} For the second implication, fix a coloring $c:[\nu]^2\to \lambda$. If $\{v_\alpha\,|\,\alpha<\beta\}\subseteq\nu$ is the increasing enumeration of the vertex-set of a size-$\mu$ graph which is highly connected in the color $i$, then $X=\{v_\alpha\,|\,\alpha<\mu\}$ is well-connected in the color $i$.
\end{proof}
\begin{lemma}\label{gch} The relation $\nu\rightarrow_{wc}(\mu)_\lambda^2$ holds for any infinite $\lambda< \text{cf}(\mu)\leq\mu<\nu$. If also $\mu^\lambda<\nu$, then the relation $\nu\rightarrow_{hc}(\mu)_\lambda^2$ holds as well.
\end{lemma}
\begin{proof} Fix a coloring $c:[\nu]^2\to\lambda$. Take then some $\xi\in\nu\backslash\mu$ and an $A\in [\mu]^\mu$ and an $i\in\lambda$ such that $c(\alpha,\xi)=i$ for all $\alpha\in A$. The set $A$ is well-connected in the color $i$. The second assertion is an immediate consequence of Proposition 2.3 of \cite{BHS}.
\end{proof}
\begin{lemma}\label{muplus} Let $\mu$ be an infinite cardinal. Then the negative relations $\mu^{+}\not\to_{hc}(3)^2_\mu$ and $\mu^{+}\not\to_{wc}(\omega_1)^2_\mu$ both hold. If there exists a $\square_\mu$-sequence on $\mu^+$, then the relation $\mu^{+}\not\to_{wc}(3)^2_\mu$ holds as well. 
\end{lemma}
\begin{proof} The following coloring is due to Erd\H{o}s and Kakutani \cite{erdos-kakutani}: for each $\beta<\mu^+$ fix a bijection $b_\beta:\beta\to |\beta|$. For $\alpha<\beta<\mu^+$ let $c(\alpha,\beta)=b_\beta(\alpha)$. Observe that $c(\alpha,\gamma)\neq c(\beta,\gamma)$ whenever $\alpha<\beta<\gamma<\mu^+$. In consequence, for each $i\in\mu$ the family of $i$-colored edges determines an acyclic graph. It follows that the only highly connected graphs that are monochromatic with respect to $c$ are of size $2$. It follows also that if $X$ is well-connected in the color $i$ then for any $\beta<\gamma$ in $X$ and $\alpha<\beta$, the colors $c(\alpha,\beta)$ and $c(\alpha,\gamma)$ cannot both be $i$. This implies that the connecting path for any such $\beta<\gamma$ in $X$ must fall within the interval $[\beta,\gamma]$, and that the the order-type of $X$ is, in consequence, at most $\omega$.

In \cite[Section 3.4]{todorcevic-walks}, Todorcevic describes a strong variant $\varrho:[\omega_1]^2\to\omega$ of the Erd\H{o}s-Kakutani coloring. The $\varrho$-monochromatic subgraphs of $[\omega_1]^2$ are acyclic in the following strong sense:
\emph{Any connected component of any $\varrho$-monochromatic subgraph of $[\omega_1]^2$ is of the form $\{\{\alpha,\beta\}\,|\,\beta\in B\}$, where $B\subseteq \omega_1\backslash(\alpha+1)$.} It follows immediately that $\omega_1\not\to_{wc}(3)_\omega^2$.

Todorcevic derives the coloring $\varrho$ from the subadditive function $\rho:[\omega_1]^2\to\omega$. The higher-cardinal functions $\rho:[\mu^+]^2\to\mu$ are again subadditive when defined with respect to a $\square_\mu$-sequence \cite[Lemma 7.3.7]{todorcevic-walks}. The italicized assertion of the previous paragraph then again holds, with $\mu^+$ in place of $\omega_1$ and $\varrho:[\mu^+]^2\to\mu\times\mu$ defined by
$$\varrho(\alpha,\beta)=(\rho(\alpha,\beta),\, \text{otp}\{\xi\leq\alpha\,|\,\rho(\xi,\alpha)\leq\rho(\alpha,\beta)\})$$
The verification is exactly as for the case of $\mu^+=\omega_1$ and therefore left to the reader. Such a $\varrho$, in conclusion, witnesses that $\mu^{+}\not\to_{wc}(3)^2_\mu$.
\end{proof}
In particular, just as for the $\to_{hc}$ arrow, $\omega_2$ is the least cardinal $\mu$ for which the relation $\mu\to_{wc}(\mu)_{\omega}^2$ may possibly hold.

By the following, either of the relations $\omega_2\to_{hc}(\omega_2)_\omega^2$ or $\omega_2\to_{wc}(\omega_2)_\omega^2$ entails large cardinal assumptions.
\begin{lemma}\label{square} If there exists a $\square(\mu)$-sequence $\vec{C}$ with the additional property that the set $\{\alpha\in\mu\,|\,\mathrm{otp}(C_\alpha)=\lambda\}$ is stationary in $\mu$, then $\mu\not\to_{wc}(\mu)_\lambda^2$.
\end{lemma}
\begin{proof} The above proposition, with $\mu\not\to_{hc}(\mu)_\lambda^2$ in place of $\mu\not\to_{wc}(\mu)_\lambda^2$ in its conclusion, is argued in \cite[Proposition 15]{BHS}. As the reader may verify, that argument applies wholesale to the relation $\to_{wc}$ as well.
\end{proof}
\begin{remark}\label{rmk} As noted in \cite{BHS}, the consistency strength of there not existing a $\square(\omega_2)$-sequence as in Lemma \ref{square} is exactly a Mahlo cardinal. Forthcoming work by Rinot and Lambie-Hanson \cite{K3} reduces the premise of Lemma \ref{square} to the existence of any $\square(\mu)$-sequence whatsoever, and appears to reduce the second premise of Lemma \ref{muplus} from $\square_\mu$ to $\square(\mu^+)$ as well. Recall from \cite{todorcevic-pairs} and \cite{jensen} that a regular uncountable cardinal $\mu$ indexes no $\square(\mu)$-sequence if and only if it is weakly compact in $\mathrm{L}$.\end{remark}
We may now more precisely describe the relations evoked in Theorem \ref{inL}:
\begin{proof}[Proof of Theorem \ref{inL}] The arrow $\to_{*}$ will simultaneously denote the three arrows $\to$ and $\to_{hc}$ and $\to_{wc}$. Throughout, the cardinal $\mu$ should be understood to be regular. We work in $\mathrm{L}$. Our assertions about square sequences existing therein are due essentially to \cite{jensen}. The equivalence of  $\nu\to(\mu)_\lambda^2$ and $\nu\to_{hc}(\mu)_\lambda^2$ 
for finite $\nu$ is definitional, as remarked above. We therefore restrict our attention below to infinite $\nu$. The nontrivial possibilities are the following:

\textbf{Case 1: $\lambda<\mu<\nu\,$}: The relations $\nu\to_{hc}(\mu)_\lambda^2$ and $\nu\to_{wc}(\mu)_\lambda^2$ both hold, by Lemma \ref{gch} and the cardinal arithmetic of $\mathrm{L}$. That same arithmetic will ensure that $\nu\to(\mu)_\lambda^2$ as well, by the Erd\H{o}s-Rado Theorem when $\mu$ is a successor cardinal, and by \cite[Theorem 17.1]{erdosetal} when $\mu$ is a limit cardinal.

\textbf{Case 2: $\lambda<\mu=\nu\,$}: If $\lambda$ is finite then $\nu\to_{hc}(\nu)_\lambda^2$ and $\nu\to_{hc}(\nu)_\lambda^2$, by Proposition \ref{finite} and Lemma \ref{positive}. If $\lambda$ is infinite, then $\nu\not\to_{*}(\nu)_\lambda^2$ for any $\nu$ which is not weakly compact, since there then exists a $\square(\nu)$-sequence as in the premise of Lemma \ref{square}. If $\nu$ is weakly compact, then clearly $\nu\to_{*}(\nu)_\lambda^2$.

\textbf{Case 3:} $\mu\leq\lambda<\nu\,$: If $\nu\neq\lambda^+$ then $\nu\to_{*}(\mu')_\lambda^2$ holds for some $\mu'>\mu$ by Case 1, hence $\nu\to_{*}(\mu)_\lambda^2$ holds as well. If $\nu=\lambda^+$ then $\nu\not\to_{wc}(3)_\lambda^2$ and $\nu\not\to_{hc}(3)_\lambda^2$, since there exists a $\square_\lambda$-sequence as in Lemma \ref{muplus}. The relation $2^\lambda\not\to(3)_\lambda^2$ is $\mathsf{ZFC}$ folklore, hence $\lambda^+\not\to(3)_\lambda^2$ in $\mathrm{L}$ as well.
\end{proof}

In contrast to Theorem \ref{inL} is the following, in which the arrows $\to_{hc}$ and $\to_{wc}$ diverge at the first place they possibly can.

\begin{theorem}\label{Mitchell1} Let $\mathbb{M}$ denote the Mitchell collapse of a weakly compact cardinal $\lambda$ to $\omega_2$. Then $$\omega_2\to_{wc}(\omega_2)_\omega^2\textnormal{ but }\omega_2\not\to_{hc}(\omega_2)_\omega^2$$ in the forcing extension of $V$ by $\mathbb{M}$.
\end{theorem}
Critical to the argument of the theorem is the following feature distinguishing $\to_{wc}$ from $\to_{hc}$. 

\begin{lemma}\label{trees} Fix a coloring $c:[\nu]^2\to\lambda$. Let $\alpha\vartriangleleft_i\beta$ if and only if $\alpha<\beta$ and $\{\alpha,\beta\}$ is well-connected in the color $i$. Then the relation $\vartriangleleft_i$ is a tree-ordering of $\nu$, and any branch of the associated tree $T_c(\vartriangleleft_i)$ is well-connected in the color $i$.
\end{lemma}
\begin{proof} Suppose that paths $p_\alpha$ and $p_\beta$ respectively witness that $\alpha\vartriangleleft_i\gamma$ and $\beta\vartriangleleft_i\gamma$ for some $\alpha<\beta<\gamma<\nu$. Then $p_\alpha\cup p_\beta$ witnesses that $\alpha\vartriangleleft_i\beta$. The rest of the assertion is immediate.
\end{proof}
However, as it is at least not \emph{a priori} evident that any of the trees $T_c(\vartriangleleft_i)$ $(i\in\lambda)$ is a $\nu$-tree (i.e., has levels all of cardinality less than $\nu$), Theorem \ref{Mitchell1} does not immediately follow from the tree property holding at $\nu$. Some more active engagement with Mitchell's argument is necessary.

\begin{proof}[Proof of Theorem \ref{Mitchell1}] It will emerge below that $V^{\mathbb{M}}\vDash\,2^{\omega}=\omega_2$. It will then follow immediately that $V^{\mathbb{M}}\vDash\,\omega_2\not\to_{hc}(\omega_2)^2_\omega$, by Proposition \ref{sierpinski}.

We therefore focus on the argument that $V^{\mathbb{M}}\vDash\,\omega_2\to_{wc}(\omega_2)_\omega^2$. We review along the way the fundamentals of the forcing $\mathbb{M}$. As noted, of the now numerous accounts of $\mathbb{M}$ available, we largely follow those of Cummings in \cite{cummings_handbook} and of Mitchell in \cite{mitchell2}. We begin, in particular, by assuming that $\lambda$ is measurable. The associated elementary embedding $j:V\to M$ with $\text{crit}(j)=\lambda$ appreciably simplifies the argument, which then concludes with the recognition that the reflection properties of a weakly compact cardinal $\lambda$ would have sufficed.

Let $\mathbb{P}_\alpha=\mathrm{Add}(\omega,\alpha)$ and let $\mathbb{P}=\mathbb{P}_\lambda$ and let $\mathbb{F}_\alpha=\mathrm{Add}(\omega_1,1)^{V^{\mathbb{P}_\alpha}}$. The conditions of $\mathbb{M}$ are the pairs $(p,f)$ for which \begin{itemize}
\item $p\in\mathrm{Add}(\omega,\lambda)$,
\item $f$ is a partial function on $\lambda$ with countable support, and
\item $f(\alpha)$ is a $\mathbb{P}_\alpha$-name for a condition in $\mathbb{F}_\alpha$.
\end{itemize}
$\mathbb{M}$ is ordered so that $(q,g)\leq_{\mathbb{M}}(p,f)$ if and only if \begin{itemize}
\item $q\leq_\mathbb{P} p$,
\item $\text{supp}(g)\supseteq\text{supp}(f)$, and
\item for all $\alpha\in\text{supp}(f)$ the restriction $q\!\restriction\!(\omega\times\alpha)$ forces that $g(\alpha)\leq_{\mathbb{F}_\alpha} f(\alpha)$.
\end{itemize}
The poset $\mathbb{M}$ appears in this presentation as a slightly odd two-step forcing, at least sufficiently so that the usual arguments apply to show that $\mathbb{M}$ is $\lambda$-c.c. In consequence,\begin{enumerate}
\item Forcing with $\mathbb{M}$ preserves cardinals $\mu\geq\lambda$.
\end{enumerate}
Consider now $\tilde{\mathbb{F}}:=\{(p,f)\in\mathbb{M}\,|\,p=\varnothing\}$, viewed as a suborder of $\mathbb{M}$. Observe that $\pi:(p,(\varnothing,f))\mapsto(p,f)$ is a projection from $\mathbb{P}\times\tilde{\mathbb{F}}$ to $\mathbb{M}$; observe also that $\mathbb{P}$ forces that $\tilde{\mathbb{F}}$ is $\omega_1$-closed. Hence any countable sequence of ordinals in $V^{\mathbb{M}}$ is in $V^{\mathbb{P}\times\tilde{\mathbb{F}}}$ and hence, by Easton's Lemma, is in $V^\mathbb{P}$. In particular,\begin{enumerate}
\item[(2)] Forcing with $\mathbb{M}$ preserves $\omega_1$.
\end{enumerate}
To see that (1) and (2) together account for all the cardinals in $V^\mathbb{M}$, consider the following alternate presentation of $\mathbb{M}$ as a forcing iteration $\langle \mathbb{S}_\alpha,\dot{\mathbb{T}}_\alpha\,|\,\alpha<\lambda\rangle$ in which, for limit ordinals $\alpha$,
\begin{itemize}
\item $\mathbb{T}_\alpha$ is a $\mathbb{S}_\alpha$-name for $\text{Add}(\omega,1)$, and 
\item $\mathbb{T}_{\alpha+1}$ is a $\mathbb{S}_{\alpha+1}$-name for $\text{Add}(\omega_1,1)$.
\end{itemize}
All other terms of the iteration are trivial. The $\text{Add}(\omega,1)$-iterands take finite supports, while the $\text{Add}(\omega_1,1)$-iterands take countable supports.

This framing invites a more dynamic view: at stages $\omega\cdot\alpha<\lambda$, the forcing $\mathbb{M}$ adds an $\alpha^{th}$ Cohen real to the $\omega\cdot\alpha^{th}$ extension of $V$, then collapses the size of the continuum to $\omega_1$. In the process, each $\alpha<\lambda$ is collapsed to $\omega_1$, hence
$$V^{\mathbb{M}}\vDash 2^{\omega}=\omega_2=\lambda$$
This framing also facilitates the factorization of $j(\mathbb{M})$ in terms of $\mathbb{M}$. Namely, write $\mathbb{M}_\alpha$ for the length-$\alpha$ initial segment of $\mathbb{M}$. Then $\mathbb{M}=j(\mathbb{M})_\lambda$, hence the map $\pi: m\mapsto m\!\restriction\!\lambda$ is a projection $j(\mathbb{M})\to\mathbb{M}$, hence any $j(\mathbb{M})$-generic filter $H$ over $V$ induces an $\mathbb{M}$-generic filter $G=\pi''H$ over $V$. Clearly $j''G\subseteq H$, hence $j:V\to M$ extends to an elementary embedding $k:V[G]\to M[H]$ (see \cite[ Proposition 9.1]{cummings_handbook}). Moreover, we may factor $M[H]$ as $M[G][K]$, where $K$ is generic over $M[G]$ with respect to some $\mathbb{Q}$ close in spirit to $\mathbb{M}_{j(\lambda)}^{M[G]}$. Again by Easton's Lemma, only ``the Cohen part'' of $j(\mathbb{M})$ adds countable sequences of ordinals, hence $M[G][K]\vDash\text{cf}(\lambda)=\omega_1$.

Now fix a coloring $c:[\lambda]^2\to\omega$ in $V[G]$. Since $^\lambda M\subseteq M$ \cite[Theorem 5.7]{kanamori} and $\mathbb{M}$ is $\lambda$-c.c., the coloring $c$ is in $M[G]$ as well. Recall the induced trees $T_c(\vartriangleleft_i)$ of Lemma \ref{trees}; write $T_c(\vartriangleleft_i)\!\restriction\!\alpha$ for the restriction of the tree-ordering $\vartriangleleft_i$ to the ordinals of $\alpha$. Observe that for each $i\in\omega$, $$k(T_c(\vartriangleleft_i))\!\restriction\!\lambda=T_{k(c)}(\vartriangleleft_i)\!\restriction\!\lambda=T_c(\vartriangleleft_i)$$ One of these trees will have a cofinal branch in $M[G][K]$. To see this, consider the function $k(c)(\,\cdot\,,\lambda):\lambda\to\omega$ in $M[G][K]$. Since $\text{cf}(\lambda)=\omega_1$ in $M[G][K]$, the function is constantly $i$ on some cofinal $A\subseteq\lambda$. The $\vartriangleleft_i$-downward closure of $A$ defines a branch $b$ as desired. By the following lemma, $b\in M[G]\subseteq V[G]$. As $c$ was arbitrary, this will complete the proof.
\begin{lemma}[\cite{mitchell2}] Let $\lambda$ be a cardinal and let $b\in M[G][K]$ be a subset of $\lambda$ such that $b\cap x\in M[G]$ for all $x\in ([\lambda]^\omega)^{M[G]}$. Then $b\in M[G]$.
\end{lemma}
To see that $b$ satisfies the assumptions of the lemma, take any $x\in([\lambda]^\omega)^{M[G]}$ and $\xi\in b\backslash(\text{sup}(x)+1)$. Then clearly $$b\cap x=\{\alpha\in x\,|\,\{\alpha,\xi\}\text{ is well-connected in the color }i\}^{M[G]}$$
is an element of $M[G]$.

Observe finally that the $\Pi_1^1$-indescribability of a weakly compact $\lambda$ would have sufficed in place of the elementary embedding $j$ above; for more concrete argumentation in that setting, the reader is referred to the original \cite{mitchell1}.
\end{proof}
\begin{corollary}\label{strength} The consistency strength of the relation $\omega_2\to_{wc}(\omega_2)_\omega^2$ is exactly a weakly compact cardinal. \end{corollary}
\begin{proof} This follows immediately from Remark \ref{rmk} and Theorem \ref{Mitchell1}.
\end{proof}
Assuming the existence of a weakly compact cardinal $\nu>\mu$, the above readily adapts to show the consistency of $[(\mu^{++}\to_{wc}(\mu^{++})_\mu^2$ while $(\mu^{++}\not\to_{hc}(\mu^{++})_\mu^2]$. Clearly the associated variant of Corollary \ref{strength} will again follow as well.

We close with the question of whether the assumption that $\mu$ is regular is needed in Theorem 0.4. In most cases it is not; the obscurity concentrates in the question of $\mu^{+}\rightarrow_{*}(\mu)^2_{\lambda}$, where $\lambda=\text{cf}(\mu)<\mu$. Here as before, the arrow $\rightarrow_{*}$ condenses the three separate questions of $\to$, $\to_{hc}$, and $\to_{wc}$. In this sense, the first question is the following:
\begin{question} Under what assumptions does $\aleph_{\omega+1}\rightarrow_{*}(\aleph_\omega)^2_{\omega}$?
\end{question}

\textbf{Acknowledgements:} This work owes its existence (but none of its faults) to the very steady stimulus of conversations with Michael Hru\v{s}\'{a}k.

\end{document}